\documentclass[12pt]{amsart}

\usepackage{amsmath}     
\usepackage{amssymb}     
\usepackage{amsthm}       
\usepackage{dsfont}          
\usepackage{graphicx}      

\newtheorem{theorem}{Theorem}[section]
\newtheorem{corollary}[theorem]{Corollary}
\newtheorem{definition}[theorem]{Definition}
\newtheorem{example}[theorem]{Example}
\newtheorem{lemma}[theorem]{Lemma}

\begin{document}

\title[Sudoku Necessary Condition]{A Necessary Solution \\Condition for Sudoku}

\author{Thomas Fischer}

\address{Fuchstanzstr. 20, 60489 Frankfurt am Main, Germany.}
%

\email{dr.thomas.fischer@gmx.de}

\date{}     

\begin{abstract}
We develop a new discrete mathematical model which includes the classical 
Sudoku puzzle, Latin Squares and gerechte designs. This problem is described 
by integer equations and a special type of inequality constraint. We consider 
solutions of this generalized problem and derive a necessary condition on these
solutions. The results are illustrated with examples.
\end{abstract}

\thanks{Thanks to Peter Schermer for providing the article of Behrens.}



\maketitle

                        %
                        %
\section{Introduction} \label{S:intro} 
A Sudoku is a square consisting of a 9$\times$9 grid which is
partly pre-populated by numbers between 1 and 9 called the givens.
The problem consists of finding numbers between 1 and 9 for all
unpopulated cells, such that each row, each column and each block 
consists of exactly the numbers $1, \ldots, 9$. The blocks of a Sudoku 
partition the Sudoku square into subsquares of size 3$\times$3. Each 
Sudoku consists of 9 rows, 9 columns and 9 blocks.

\begin{example}
A lot of examples for Sudoku are spread over the mathematical literature
and you can also find some in your daily newspaper or spread over the internet. 
We refer here to a typical Sudoku square of Delahaye \cite{Del} 
depicted in Fig. \ref{Fig1}.
\begin{figure}
\includegraphics{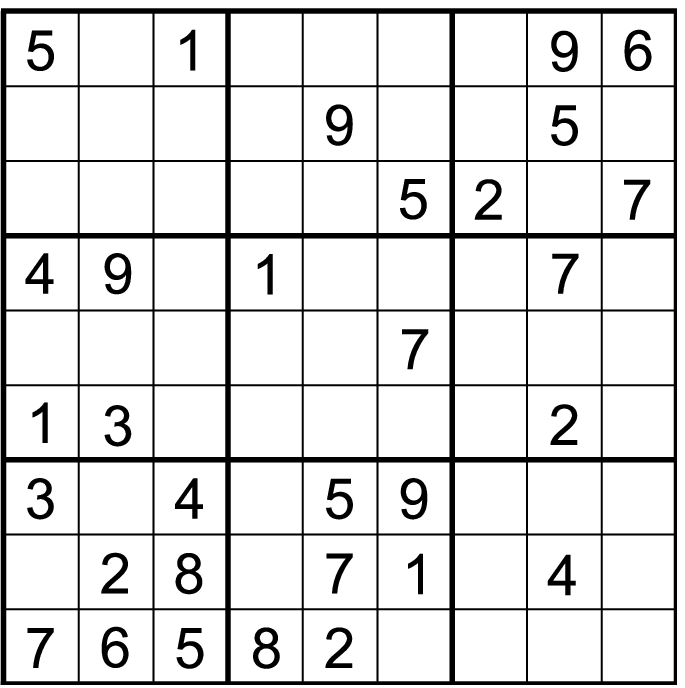}
\caption{A typical Sudoku} 
\label{Fig1}
\end{figure}
\end{example}

In Section~\ref{S:model} we introduce a mathematical model based on an 
integer constraint system, which we call the generalized Sudoku problem.
In Section~\ref{S:sudoku} the Sudoku puzzle is reformulated as a special case 
of this problem.

The mathematical model depends on an integer $n \ge 2$ and is more general 
than a Sudoku puzzle. Sudoku problems of arbitrary size are covered by this model. 
The present approach also covers Latin Squares (see D\'{e}nes and Keedwell \cite{DK1}) 
and ``gerechte Anordnungen",  introduced by Behrens \cite{Beh}.
The latter is usually termed as gerechte designs (see Vaughan 
\cite{Vau}) in the literature.

The main result in Section~\ref{S:neccond} states a necessary condition on the
solutions of the generalized Sudoku problem. The proof of this result is based
on algebraic properties and does not involve a separation theorem.

The use of necessary conditions 
is a powerful tool in the investigation of nonlinear optimization problems (see 
Fiacco and McCormick \cite{FC} or Luenberger \cite{Luen}). Beside optimization 
theory the Gaussian elimination process for solving linear systems of equations 
starts with the assumption of an existing solution and derives necessary 
conditions on this solution. In the theory of discrete optimization usually other 
techniques like relaxation or cutting planes are used (see Nemhauser and 
Wolsey \cite{NW}).

There exist several papers dealing with algorithms for solving Sudoku. 
Classic approaches are brute force methods or paper-and-pencil methods 
(Crook \cite{Cro}). It is also possible to use branch-and-cut algorithms (see Kaibel and 
Koch \cite{KK}) or the Algorithm X implementing dancing links described by Knuth \cite{Knu}. 

The present paper describes a Sudoku puzzle as an integer programming problem. Similar 
approaches can be found in the literature (see Fontana \cite{Fon}). To the best of my 
knowledge there do not exist results or methods based on necessary solution conditions, so far.

Finally, we collect some basic terms and notations. Let $\mathds{Z}$ denote the set of integers
and let $\mathds{Z}^n$ denote the $n$-times cartesian product. The elements of a set are called 
distinct if they are pairwise distinct, i.e., each two elements of the set are not equal. The vectors 
$\mathbf{0}$ respectively $\mathbf{1}$ denote the zero respectively one vector, consisting only of 
zeros respectively ones in each component.
The number of components is often indicated by an index. Each vector is considered to be a column vector.
$U$ denotes the identity matrix. The transpose of a vector or a matrix is indicated by a superscript $”T”$.
We allow explicitly the use of vectors respectively matrices of dimension zero, i.e., without components 
respectively entries. In this case we call the matrix the empty matrix.
The sign function is denoted by $sgn$ and is generalized in Definition \ref{D:sgn} to vectors. 
We consider the sum over an empty index set to be zero.

                        %
                        %
\section{The Mathematical Model} \label{S:model}
Let $n$ be an integer with $n \ge 1$. We define the sum
\[                                                         
s(n)  = \sum_{i=1}^{n-1}i
\]
and define a matrix $A(n)$ with $s(n)$ rows and $n$ columns inductively. 
For $n=1$, let $A(1)$ 
denote the empty matrix, i.e., a matrix without entries.
Assume the matrix $A(n-1)$ had been defined with $s(n-1)$ rows 
and $n-1$ columns. Then we set
{
\renewcommand{\arraystretch}{2.0}     
\[
A(n) = 
\begin{pmatrix}    
\begin{array}{c|c}
\mathbf{1}_{n-1} & - U_{n-1} \\
\hline
\mathbf{0}_{s(n-1)} & A(n-1)
\end{array}
\end{pmatrix}.
\]
}

\begin{example}             \label{E:1}
The matrix $A(n)$ for $n=9$ is depicted in Fig. \ref{Fig2}.
\end{example}

We state two elementary properties of the matrix $A(n)$.

1. The vector of ones $\mathbf{1}_{n}=(1, \ldots, 1)^T \in \mathds{Z}^n$ is mapped 
to the zero-vector by $A(n)$, i.e., $A(n)\mathbf{1}_{n} = \mathbf{0}_{s(n)}$.

2. The matrix $A(n)$ has rank $n-1$.

We extend the matrix $A(n)$ to a matrix $A$ with $n \cdot s(n)$ rows 
and $n^2$ columns. The matrix $A$ consists in the ``main diagonal" of $n$ 
matrices $A(n)$ and the remaining values are set to zero. The matrix $A$ 
depends on the value $n$, but we do not state this dependence explicitly.

{
\renewcommand{\arraystretch}{2.0}     
\[
A = 
\begin{pmatrix}    
\begin{array}{c|c|p{1cm} c p{1cm}|c}
A(n) & 0 & & \hdots & & 0 \\
\hline
0 & A(n) & & \hdots&  & 0 \\
\hline
& & & & & \\
\vdots & \vdots & & \ddots &  & \vdots \\
& & & & & \\
\hline
0 & 0 & & \hdots & & A(n)
\end{array}
\end{pmatrix}.
\]
}

Given the set $\{1, \ldots, n^2\} \subset \mathds{Z}$, let $\pi$ be any
permutation on this set, i.e., 
\[
\pi: \{1, \ldots, n^2\} \longrightarrow \{1, \ldots, n^2\}
\]
be a permutation. We extend the notion of permutation to the matrix $A$,  
i.e., we define $\pi (A) = (a^{\pi^{-1}(1)}, \ldots, a^{\pi^{-1}(n^2)})$, 
where $a^j$ denotes the $j^{th}$ column of $A$ for $j = 1, \ldots , n^2$. 
Analoguously we define the term 
$\pi(x) =  (x_{\pi^{-1}(1)}, \ldots, x_{\pi^{-1}(n^2)})^T$, for any 
$x = (x_1, \ldots, x_{n^2})^T \in \mathds{Z}^{n^2}$.
We define the matrix $A_\pi = \pi (A)$, i.e., we interchange the columns of $A$ 
according to the permutation $\pi$.

\begin{figure}
{
\renewcommand{\arraystretch}{0.5}     
\small                                                 
\[
\begin{bmatrix}    
1 & -1 & 0 & 0 & 0 & 0 & 0 & 0 & 0\\
1 & 0 & -1 & 0 & 0 & 0 & 0 & 0 & 0\\
1 & 0 & 0 & -1 & 0 & 0 & 0 & 0 & 0\\
1 & 0 & 0 & 0 & -1 & 0 & 0 & 0 & 0\\
1 & 0 & 0 & 0 & 0 & -1 & 0 & 0 & 0\\
1 & 0 & 0 & 0 & 0 & 0 & -1 & 0 & 0\\
1 & 0 & 0 & 0 & 0 & 0 & 0 & -1 & 0\\
1 & 0 & 0 & 0 & 0 & 0 & 0 & 0 & -1\\

0 & 1 & -1 & 0 & 0 & 0 & 0 & 0 & 0\\
0 & 1 & 0 & -1 & 0 & 0 & 0 & 0 & 0\\
0 & 1 & 0 & 0 & -1 & 0 & 0 & 0 & 0\\
0 & 1 & 0 & 0 & 0 & -1 & 0 & 0 & 0\\
0 & 1 & 0 & 0 & 0 & 0 & -1 & 0 & 0\\
0 & 1 & 0 & 0 & 0 & 0 & 0 & -1 & 0\\
0 & 1 & 0 & 0 & 0 & 0 & 0 & 0 & -1\\

0 & 0 & 1 & -1 & 0 & 0 & 0 & 0 & 0\\
0 & 0 & 1 & 0 & -1 & 0 & 0 & 0 & 0\\
0 & 0 & 1 & 0 & 0 & -1 & 0 & 0 & 0\\
0 & 0 & 1 & 0 & 0 & 0 & -1 & 0 & 0\\
0 & 0 & 1 & 0 & 0 & 0 & 0 & -1 & 0\\
0 & 0 & 1 & 0 & 0 & 0 & 0 & 0 & -1\\

0 & 0 & 0 & 1 & -1 & 0 & 0 & 0 & 0\\
0 & 0 & 0 & 1 & 0 & -1 & 0 & 0 & 0\\
0 & 0 & 0 & 1 & 0 & 0 & -1 & 0 & 0\\
0 & 0 & 0 & 1 & 0 & 0 & 0 & -1 & 0\\
0 & 0 & 0 & 1 & 0 & 0 & 0 & 0 & -1\\

0 & 0 & 0 & 0 & 1 & -1 & 0 & 0 & 0\\
0 & 0 & 0 & 0 & 1 & 0 & -1 & 0 & 0\\
0 & 0 & 0 & 0 & 1 & 0 & 0 & -1 & 0\\
0 & 0 & 0 & 0 & 1 & 0 & 0 & 0 & -1\\

0 & 0 & 0 & 0 & 0 & 1 & -1 & 0 & 0\\
0 & 0 & 0 & 0 & 0 & 1 & 0 & -1 & 0\\
0 & 0 & 0 & 0 & 0 & 1 & 0 & 0 & -1\\

0 & 0 & 0 & 0 & 0 & 0 & 1 & -1 & 0\\
0 & 0 & 0 & 0 & 0 & 0 & 1 & 0 & -1\\

0 & 0 & 0 & 0 & 0 & 0 & 0 & 1 & -1
\end{bmatrix}
\]
}
\caption{The matrix $A(9)$} 
\label{Fig2}
\end{figure}

\begin{definition} \label{D:nonzero}
Let $s \ge 1$. For any point $y = (y_1, \ldots, y_s)^T \in \mathds{Z}^s$ 
we write $y < > \mathbf{0}$ if each component of $y$ is nonzero, 
i.e., if $y_i \ne 0$ for $i=1, \ldots, s$.
\end{definition}

This definition should not be confused with the expression $y \ne \mathbf{0}$, 
where only one component of $y$ has to be nonzero.

Now we are in position to state the generalized Sudoku problem. Given is $n \ge 2$, 
some permutations $\pi_1$, $\pi_2$, $\pi_3$ on $\{1, \ldots, n^2\}$, 
some $0 \le k \le n^2$, an index set $\{i_1, \ldots, i_k \} \subset \{ 1, \ldots, n^2 \}$ 
and givens $g_{i_1}, \ldots , g_{i_k} \in\mathds{Z}$ with $1 \le g_{i_l} \le n$ for 
$l = 1, \ldots, k$.

The generalized Sudoku problem consists in finding $x = (x_1, \ldots, x_{n^2})^T$ in
$\mathds{Z}^{n^2}$, such that $1 \le x_i \le n$ for $i = 1, \ldots , n^2$, 
\[
\begin{array}{lll}
A_{\pi_1} x & <> & \mathbf{0}, \\
A_{\pi_2} x & <> & \mathbf{0}, \\
A_{\pi_3} x & <> & \mathbf{0},
\end{array}
\]
and $x_{i_l} = g_{i_l}$ for $l = 1, \ldots, k$.

We will investigate the relation between this model and the Sudoku puzzle in detail in the next section.

Due to the prescribed givens some of the components of $x$ are known in advance. Therefore it 
is possible to reduce the size of this problem from $n$ to $n-k$ variables. This is advantageous for 
numerical purposes, but may interfere theoretical properties. We decided not to eliminate these
variables in our considerations and maintain a distinction between equality and inequality constraints.

This general approach includes gerechte designs (see Vaughan \cite{Vau}).
In particular the case $\pi_3 = \pi_2$ is allowed and in this case the generalized 
Sudoku problem describes Latin Squares (see D\'{e}nes and Keedwell \cite{DK1}).

The generalized Sudoku problem can be classified as a nonlinear integer problem. 
This problem even possesses nonconvex constraints. This contrasts with 
famous integer optimization problems like the traveling salesman problem and 
the knapsack problem, which usually are modeled as linear integer problems 
(see Papadimitriou and Steiglitz \cite{PS}).

                        %
                        %
\section{Modeling Sudoku} \label{S:sudoku}
In this section we show, that a classical Sudoku puzzle can be refomulated as a generalized 
Sudoku problem. We state an elementary result for later use without proof.

\begin{lemma} \label{L:1}
Let $n \geq 2$ and $x_1, \ldots, x_n$ be integers. The following statements are equivalent:\\
(i) $x_1, \ldots , x_n$ are distinct and $1 \leq x_i \leq n$ for $i = 1, \ldots, n$. \\
(ii) $\{x_1, \ldots , x_n\} = \{1, \ldots , n\}$.
\end{lemma}

\begin{lemma} \label{L:2}
Let $n \ge 2$ and $x=(x_1, \ldots, x_n)^T \in \mathds{Z}^n$. Then $A(n)x <> \mathbf{0}$ 
if and only if the components $x_1, \ldots, x_n$ of $x$ are distinct.
\end{lemma}
\begin{proof}
We show the claim by induction on $n$. The claim is true for $n=2$, since
$A(2)x = x_1 - x_2$ and this expression is nonzero if and only if $x_1$ and $x_2$ are distinct.

Assume the claim holds for $n-1$ and consider $x=(x_1, \ldots, x_n)^T$. By induction assumption, 
$A(n-1)(x_2, \ldots, x_n)^T <> \mathbf{0}$ if and only if $x_2, \ldots, x_n$ are distinct.
Obviously $x_1, \ldots, x_n$ are distinct if and only if $x_1 - x_i \ne 0$ for 
$i = 2, \ldots, n$ and $x_2, \ldots, x_n$ are distinct. The condition 
$x_1 - x_i \ne 0$ for $i = 2, \ldots, n$ can be expressed as
\[
\begin{pmatrix}
\begin{array}{c|c}
\mathbf{1}_{n-1} & - U_{n-1}
\end{array}
\end{pmatrix}
\begin{pmatrix} 
x_1 \\ \vdots \\ x_n
\end{pmatrix}
<> \mathbf{0}.
\]
By definition of $A(n)$ the claim is proved.
\end{proof}

We combine the Lemma \ref{L:1} and \ref{L:2} and obtain:

\begin{lemma} \label{L:3}
Let $n \ge 2$ and let $x=(x_1, \ldots, x_n)^T \in \mathds{Z}^n$. The following statements are equivalent: \\
(i) $A(n)x <> \mathbf{0}$ and $1 \le x_i \le n$ for $i=1, \ldots, n$. \\
(ii) $\{x_1, \ldots, x_n\} = \{1, \ldots, n\}$.
\end{lemma}

We introduce the notion of an $x$-tableau, where the components
$x_1, \ldots, x_{n^2}$ of points $x \in \mathds{Z}^{n^2}$
are arranged row-wise from the left in an ${n \times n}$ square.

{
\renewcommand{\arraystretch}{2.0}     
\[    
\begin{array}{|c|c|p{1cm} c p{1cm}|c|}    
\hline
x_1 & x_2 & & \hdots & & x_n \\
\hline
x_{n+1} & x_{n+2} & & \hdots&  & x_{2 \cdot n} \\
\hline
& & & & & \\
\vdots & \vdots & & \ddots &  & \vdots \\
& & & & & \\
\hline
x_{(n-1) \cdot n+1} & x_{(n-1) \cdot n+2}  & & \hdots & & x_{n^2} \\
\hline
\end{array}
\]
}
In the sequel we do not distinguish between the $x$-tableau consisting of the components of 
$x$ and an $x$-tableau consisting of the corresponding indices.

Let $\pi_1$ be the identical permutation on $\{1, \ldots , n^2\}$, i.e.,  $\pi_1(i) = i$ for 
$i=1, \ldots, n^2$. If we arrange a point $x=(x_1, \ldots, x_{n^2})^T \in 
\mathds{Z}^{n^2}$, such that $1 \le x_i \le n$ for $i = 1, \ldots , n^2$, in 
an $x$-tableau, then $A_{\pi_1} x < > \mathbf{0}$ holds if and only if the values in each row 
of the $x$-tableau consist of $1, \ldots, n$. This follows from Lemma \ref{L:3}
and the definition of $A_{\pi_1}$.

From the next lemma we need part (i) in this section and part (ii) in a later section.

\begin{lemma} \label{L:4}
For any permutation $\pi$ on $\{1, \ldots, n^2\}$ holds: \\
(i) $A_\pi x = A \pi^{-1}(x)$ for each $x \in \mathds{Z}^{n^2}$ and \\
(ii) $A^T_\pi \lambda = \pi ( A^T \lambda )$ for each $\lambda \in \{ -1, +1\}^{n \cdot s(n)}$.
\end{lemma}
\begin{proof} 
(i) Let $x = (x_1, \ldots, x_{n^2})^T \in \mathds{Z}^{n^2}$, then
\[
A_\pi x = \sum^{n^2}_{j=1} a^{\pi^{-1}(j)}x_j 
= \sum^{n^2}_{j=1} a^j x_{\pi(j)} = A \pi^{-1}(x).
\]
(ii) Let $\lambda \in \{-1,+1\}^{n \cdot s(n)}$, then
\[
A^T_\pi \lambda = 
\begin{pmatrix} 
a^{\pi^{-1}(1)} \lambda \\
\vdots \\
a^{\pi^{-1}(n^2)} \lambda
\end{pmatrix}
= \pi
\begin{pmatrix} 
a^1 \lambda \\ 
\vdots \\
 a^{n^2} \lambda
\end{pmatrix}
= \pi (A^T \lambda ).
\]
\end{proof}

Let $\pi_2$ be  the permutation, which maps the rows of the $x$-tableau to the columns of the $x$-tableau, i.e.,
\[
\begin{array}{lll}
\pi_2 (1) & = & 1 \\
\pi_2 (2) & = & n + 1 \\
\pi_2 (3) & = & 2 \cdot n + 1 \\
\vdots & & \\
\pi_2 (n) & = & (n-1) \cdot n + 1 \\
\pi_2 (n+1) & = & 2\\
\hdots & &.
\end{array}
\]

Using Lemma \ref{L:4}, if we arrange a point 
$x=(x_1, \ldots, x_{n^2})^T \in \mathds{Z}^{n^2}$,
such that $1 \le x_i \le n$ for $i = 1, \ldots , n^2$, in an $x$-tableau, 
then $A_{\pi_2} x = A \pi_2^{-1} (x) < > \mathbf{0}$ holds if and only if the values 
in each column of the $x$-tableau consist of $1, \ldots, n$.

If $n$ is a square number the $x$-tableau can be covered
by $n$ non-overlapping subsquares of size $\sqrt{n} \times \sqrt{n}$ and 
we number these subsquares row-wise from the left.
Let $\pi_3$ be the permutation, which maps
the $i$-th row of the $x$-tableau to the $i$-th subsquare for $i=1, \ldots, n$, i.e.,
\[
\begin{array}{lll}
\pi_3 (1) & = & 1 \\
\vdots & & \\
\pi_3 (\sqrt{n}) & = & \sqrt{n}\\
\pi_3 (\sqrt{n} + 1) & = & n + 1\\
\vdots & & \\
\pi_3 (2 \cdot \sqrt{n}) & = & n + \sqrt{n}\\
\vdots & & \\
\pi_3 (n) & = & (\sqrt{n} - 1) \cdot n + \sqrt{n}\\
\hdots & &.
\end{array}
\]
Using Lemma \ref{L:4}, if we arrange a point 
$x=(x_1, \ldots, x_{n^2})^T \in \mathds{Z}^{n^2}$,
such that $1 \le x_i \le n$ for $i = 1, \ldots , n^2$, in an $x$-tableau, 
then $A_{\pi_3} x = A \pi_3^{-1} (x) < > \mathbf{0}$ holds 
if and only if the values in each $\sqrt{n} \times \sqrt{n}$ subsquare of the 
$x$-tableau consist of $1, \ldots, n$.

If $n = 9$ and $\pi_1$, $\pi_2$, $\pi_3$ are defined as before the generalized Sudoku problem is 
equivalent to the classical Sudoku problem. Each solution of the original Sudoku problem solves the 
generalized Sudoku problem and vice-versa.

Also the Sudoku puzzle can be modeled as a linear integer problem (see Kaibel and Koch \cite{KK} 
and Provan \cite{Pro}).  Kaibel and Koch \cite{KK} proposed a model where the Sudoku problem 
is described by 0-1-variables and gave a short overview on important research results without stating 
results explicitly. They also described how to apply existing software packages to their problem successfully.
The same model had been considered by Provan \cite{Pro}, who proved a unicity result.

Another model is to formulate a Sudoku puzzle as an exact cover problem. The matrix describing 
this problem is the transpose of the matrix considered by Kaibel and Koch \cite{KK} and Provan \cite{Pro}. 

                        %
                        %
\section{Preliminaries} \label{S:prelim}
In this section we start with a generalized sign function based on the classical sign function, 
also denoted by $sgn$.

\begin{definition}     \label{D:sgn}
Let $s \ge 1$. For any point $y = (y_1, \ldots, y_s)^T \in \mathds{Z}^s$ 
with $y<> \mathbf{0}$ we define the generalized sign function 
$sgn: \mathds{Z}^s \longrightarrow \mathds{Z}^s $ by
\[
sgn ( y ) =
\begin{pmatrix}
sgn ( y_1 ) \\
\vdots \\
sgn ( y_s )
\end{pmatrix}.
\]
\end{definition}

We return to the matrix $A(n)$.

\begin{lemma}   \label{L:5}
Let $n \ge 2$ and $x = (x_1, \ldots, x_n)^T \in \mathds{Z}^n$ with 
$A(n) x <> \mathbf{0}$. The point 
$A^T (n) sgn(A(n) x) \in \mathds{Z}^n$ consists of the components
\[
- \sum_{j=1}^{i-1} sgn( x_j - x_i ) + \sum_{j=i+1}^{n} sgn ( x_i - x_j )
\]
for $i=1, \ldots, n$.
\end{lemma}
\begin{proof} 
We prove this lemma by induction on $n$. For $n=2$ and $x = (x_1, x_2) \in \mathds{Z}^2$ 
with $x_1 \ne x_2$,
\[
A^T (2) sgn ( A(2)x ) = 
\begin{pmatrix} 
1 \\ 
-1
\end{pmatrix}
sgn (x_1 - x_2 ) =
\begin{pmatrix} 
sgn (x_1 - x_2 ) \\ 
- sgn ( x_1 - x_2 )
\end{pmatrix}.
\]
On the other side
\[
- \sum_{j=1}^{0}sgn (x_j - x_1 ) + \sum_{j=2}^{2} sgn (x_1 - x_j) = sgn ( x_1 - x_2 )
\]
and
\[
- \sum_{j=1}^{1} sgn (x_j - x_2) + \sum_{j=3}^{2} sgn (x_2 - x_j ) = - sgn (x_1 - x_2),
\]
which proves the claim for $n=2$. 

Assume the claim is true for $n-1$. Let $x = (x_1, \ldots, x_n)^T \in \mathds{Z}^n$ 
with $A(n) x <> \mathbf{0}$. The induction assumption reads as
\[
A^T(n-1) sgn (A(n-1)
\begin{pmatrix}
x_2 \\
\vdots \\
x_n
\end{pmatrix}
)
\]
\[
=
\begin{pmatrix} 
                                                                      &    & \sum\limits_{j=3}^{n} sgn (x_2 - x_j)  \\
- \sum\limits_{j=2}^{2} sgn (x_j - x_3)            & + & \sum\limits_{j=4}^{n} sgn (x_3 - x_j)  \\
& \vdots & \\
- \sum\limits_{j=2}^{n-2} sgn (x_j - x_{n-1}) & + & \sum\limits_{j=n}^{n} sgn (x_{n-1} - x_j)  \\
- \sum\limits_{j=2}^{n-1} sgn (x_j - x_n)       &     &
\end{pmatrix}.
\]
By definition of $A(n)$
\[
sgn ( A(n) x ) =
\begin{pmatrix} 
sgn (x_1 - x_2) \\
\vdots \\
sgn ( x_1 - x_n ) \\
sgn (A(n-1)
\begin{pmatrix}
x_2 \\
\vdots \\
x_n
\end{pmatrix}
)
\end{pmatrix}
\]
with $n-1$ components $sgn (x_1 - x_i )$ for $i = 2, \ldots, n$ and a vector 
$sgn(A(n-1)(x_2, \ldots, x_n)^T)$ with $s(n-1)$ components. Again by definition of $A(n)$
{
\renewcommand{\arraystretch}{1.5}     
\[
A^T (n) = 
\begin{pmatrix}    
\begin{array}{c|c}
\mathbf{1}_{n-1}^T & \mathbf{0}_{s(n-1)}^T \\
\hline
- U_{n-1}^T & A^T(n-1)
\end{array}
\end{pmatrix}.
\]
}
We obtain
\[
A^T(n) sgn (A(n) x )
\]
{
\renewcommand{\arraystretch}{1.5}     
\[
=
\begin{pmatrix} 
\sum\limits_{j=2}^{n} sgn (x_1 - x_j)  \\
\hline
-
\begin{pmatrix}
sgn (x_1 - x_2 ) \\
\vdots \\
sgn (x_1 - x_n )
\end{pmatrix}
+ A^T(n-1) sgn (A(n-1)
\begin{pmatrix}
x_2 \\
\vdots \\
 x_n
\end{pmatrix}
)
\end{pmatrix}
\]
}
\[
=
\begin{pmatrix} 
                                &    & \sum\limits_{j=2}^{n} sgn (x_1 - x_j)  & & \\
- sgn (x_1 - x_2)        &   &                                                          & + & \sum\limits_{j=3}^{n} sgn (x_2 - x_j)  \\
- sgn (x_1 - x_3)        & - & \sum\limits_{j=2}^{2} sgn (x_j - x_3) & + & \sum\limits_{j=4}^{n} sgn (x_3 - x_j)  \\
& & \vdots & & \\
- sgn (x_1 - x_{n-1}) & - & \sum\limits_{j=2}^{n-2} sgn (x_j - x_{n-1} ) & + & \sum\limits_{j=n}^{n} sgn (x_{n-1} - x_j)  \\
- sgn (x_1 - x_n)       & - & \sum\limits_{j=2}^{n-1} sgn (x_j - x_{n} ) &  &
\end{pmatrix}
\]
\[
=
\begin{pmatrix} 
                                &     & \sum\limits_{j=2}^{n} sgn (x_1 - x_j)  \\
- sgn (x_1 - x_2)        & + & \sum\limits_{j=3}^{n} sgn (x_2 - x_j) \\
- \sum\limits_{j=1}^{2} sgn (x_j - x_3) & + & \sum\limits_{j=4}^{n} sgn (x_3 - x_j)  \\
& \vdots & \\
- \sum\limits_{j=1}^{n-2} sgn (x_j - x_{n-1} ) & + & \sum\limits_{j=n}^{n} sgn (x_{n-1} - x_j)  \\
- \sum\limits_{j=1}^{n-1} sgn (x_j - x_{n} ) &  &
\end{pmatrix}
\]
and this shows the claim.
\end{proof} 

\begin{lemma}   \label{L:6}
Let $x = (x_1, \ldots, x_n)^T \in \mathds{Z}^n$ with $\{x_1, \ldots, x_n\} = \{ 1, \ldots, n \}$. Then
\[
- \sum_{j=1}^{i-1} sgn( x_j - x_i ) + \sum_{j=i+1}^{n} sgn ( x_i - x_j ) = 2 x_i - (n+1)
\]
for $i=1, \ldots, n$.
\end{lemma}
\begin{proof} 
We  prove this lemma by induction on $n$. For $n=1$ both sides of the equation equal zero. 
This proves the claim for $n=1$.

Assume $n \ge 2$ and the claim is true for $n-1$, i.e., for each 
$(x_1, \ldots, x_{n-1})^T \in \mathds{Z}^{n-1}$ 
with $\{x_1, \ldots, x_{n-1}\} = \{ 1, \ldots, n-1 \}$ we have
\[
- \sum_{j=1}^{i-1} sgn( x_j - x_i ) + \sum_{j=i+1}^{n-1} sgn ( x_i - x_j ) = 2 x_i - n
\]
for $i=1, \ldots, n-1$.

Let $x \in \mathds{Z}^n$ with $\{x_1, \ldots, x_n\} = \{1, \ldots, n\}$. Select the index 
$1 \le i_0 \le n$, such that $x_{i_0} = n$ and define $x^\prime \in \mathds{Z}^{n-1}$ by
\[
x^\prime = (x_1, \ldots, x_{i_0 - 1}, x_{i_0 + 1}, \ldots, x_n)^T.
\]
Let $i \in \{1, \ldots, n\}$ and we distinguish 3 cases.

Case 1: $i < i_0$. Then
\begin{align*}    
   & - \sum_{j=1}^{i-1} sgn ( x_j - x_i ) + \sum_{j=i+1}^{n} sgn ( x_i - x_j) \\
= &- \sum_{j=1}^{i-1} sgn ( x_j - x_i ) \\
    &+ \sum_{j=i+1}^{i_0 - 1} sgn ( x_i - x_j) + sgn (x_i - x_{i_0} ) + \sum_{j=i_0+1}^{n} sgn ( x_i - x_j) \\
= &- \sum_{j=1}^{i-1} sgn ( x^\prime_j - x^\prime_i ) + \sum_{j=i+1}^{n-1} sgn ( x^\prime_i - x^\prime_j) - 1 \\
= &2 x^\prime_i - n - 1 \\
= &2 x_i - (n + 1).
\end{align*}

Case 2: $i > i_0$. Then we set $l = i-1$ and obtain
\begin{align*}
   & - \sum_{j=1}^{i-1} sgn ( x_j - x_i ) + \sum_{j=i+1}^{n} sgn ( x_i - x_j) \\
= &- \sum_{j=1}^{i_0-1} sgn ( x_j - x_i ) - sgn (x_{i_0} - x_i) - \sum_{j=i_0 + 1}^{i-1} sgn ( x_j - x_i ) \\
   &+ \sum_{j=i+1}^{n} sgn ( x_i - x_j) 
\allowdisplaybreaks   \\    
= &- \sum_{j=1}^{i-2} sgn ( x^\prime_j - x^\prime_{i-1} ) + \sum_{j=i}^{n-1} sgn ( x^\prime_{i-1} - x^\prime_j) - 1 
\allowdisplaybreaks   \\
= &- \sum_{j=1}^{l-1} sgn ( x^\prime_j - x^\prime_l ) + \sum_{j=l+1}^{n-1} sgn ( x^\prime_l - x^\prime_j) - 1 
\allowdisplaybreaks   \\
= &2 x^\prime_{l} - n - 1 \allowdisplaybreaks   \\
= &2 x^\prime_{i-1} - n - 1 \\
= &2 x_i - (n + 1).
\end{align*}

Case 3: $i = i_0$. Then $x_i = x_{i_0} = n$ and
\begin{align*}
   & - \sum_{j=1}^{i-1} sgn ( x_j - x_i ) + \sum_{j=i+1}^{n} sgn ( x_i - x_j) 
\allowdisplaybreaks   \\
= &- \sum_{j=1}^{i-1} sgn ( x_j - n ) + \sum_{j=i+1}^{n} sgn ( n - x_j) 
\allowdisplaybreaks   \\
= &(i-1) + (n-i) 
\allowdisplaybreaks   \\
= &n - 1 \\
= &2 n - (n + 1) \\
= &2 x_i - (n + 1),
\end{align*}
and this completes the induction proof.
\end{proof} 

We combine the Lemma \ref{L:5} and \ref{L:6} with Lemma \ref{L:1} and \ref{L:2} and obtain

\begin{lemma}   \label{L:7}
Let $n \ge 2$ and $x = (x_1, \ldots, x_n)^T \in \mathds{Z}^n$ with components 
$1 \le x_i \le n$ for $i = 1, \ldots, n$ and let $A(n) x <> \mathbf{0}$. Then
\[
x = \frac{1}{2} ( A^T (n) sgn (A(n) x) + (n+1) \mathbf{1}_n ).
\]
\end{lemma}
\begin{proof} 
Because of Lemma \ref{L:1} and \ref{L:2} we are able to combine the Lemma 
\ref{L:5} and \ref{L:6} and we obtain
\[
A^T(n) sgn (A(n) x) = 2 x - (n+1) \mathbf{1}_n
\]
and the rest is elementary arithmetic.
\end{proof} 

\begin{example}
We consider the case $n=9$. The matrix $A(9)$ is depicted in Fig. \ref{Fig2} and 
we choose the point $x =(2, 8, 1, 5, 9, 4, 6, 3, 7)^T$. Then $sgn( A(9) x )$
\[
= ( -1, 1, -1, -1, -1, -1, -1, -1, 1, 1, -1, 1, 1, 1, 1, -1, -1, -1,
\]
\[
 -1, -1, -1, -1, 1, -1, 1, -1, 1, 1, 1, 1, -1, 1, -1, 1, -1, -1)^T
\]
and finally
\[
\frac{1}{2} ( A^T (9) sgn( A(9) x ) + 10 \cdot \mathbf{1}_{9} )
\]
\[
= \frac{1}{2} (-6, 6, -8, 0, 8, -2, 2, -4, 4)^T + 5 \cdot \mathbf{1}_{9} = (2, 8, 1, 5, 9, 4, 6, 3, 7)^T = x .
\]
\end{example}

                        %
                        %
\section{The Necessary Condition} \label{S:neccond}
In this section we extend the condition on the matrix $A(n)$ in Section \ref{S:prelim} to the 
``large" matrices $A_\pi$ (see remarks prior to Definition \ref{D:nonzero}) for any permutation $\pi$ on 
$\{1, \ldots, n^2\}$.

\begin{theorem}      \label{T:1}
Let $n \ge 2$, let $\pi$ be any permutation on $\{1, \ldots, n^2\}$ and let 
$x = (x_1, \ldots, x_{n^2})^T \in \mathds{Z}^{n^2}$ with components 
$1 \le x_i \le n$ for $i = 1, \ldots, n^2$ and $A_\pi x <> \mathbf{0}$. Then
\[
x = \frac{1}{2} ( A^T_\pi sgn (A_\pi x) + (n+1) \mathbf{1}_{n^2} ).
\]
\end{theorem}
\begin{proof}
In the first step we prove the theorem for $A_\pi =A$, i.e., $\pi = id$, the identity
mapping. Let $x = (x_1, \ldots, x_{n^2})^T \in \mathds{Z}^{n^2}$ 
with components $1 \le x_i \le n$ for $i = 1, \ldots, n^2$ and $A x <> \mathbf{0}$.
Using the diagonal structure of $A$ and Lemma \ref{L:7}, we obtain
\[
sgn ( A x ) =
\begin{pmatrix} 
sgn ( A (n)
\begin{pmatrix}
x_1 \\
\vdots \\
x_n
\end{pmatrix} ) \\
\vdots \\
sgn ( A (n)
\begin{pmatrix}
x_{(n-1) \cdot n + 1} \\
\vdots \\
x_{n^2}
\end{pmatrix} )
\end{pmatrix}
\]
and
\[
A^T sgn ( A x ) =
\begin{pmatrix} 
A^T(n) sgn ( A (n)
\begin{pmatrix}
x_1 \\
\vdots \\
x_n
\end{pmatrix} ) \\
\vdots \\
A^T(n) sgn ( A (n)
\begin{pmatrix}
x_{(n-1) \cdot n + 1} \\
\vdots \\
x_{n^2}
\end{pmatrix} )
\end{pmatrix}
\]
\newline
\[
=
\begin{pmatrix} 
2
\begin{pmatrix}
x_1 \\
\vdots \\
x_n
\end{pmatrix} 
- (n+1) \mathbf{1}_n \\
\vdots \\
2
\begin{pmatrix}
x_{(n-1) \cdot n + 1} \\
\vdots \\
x_{n^2}
\end{pmatrix}
- (n+1) \mathbf{1}_n
\end{pmatrix}
= 2 x - (n+1) \mathbf{1}_{n^2}.
\]
This shows the statement for $A_\pi = A$.

Now let $\pi$ be an arbitrary permutation on $\{1, \ldots, n^2 \}$. Let 
$x = (x_1, \ldots, x_{n^2})^T \in \mathds{Z}^{n^2}$ 
with components $1 \le x_i \le n$ for $i = 1, \ldots, n^2$ and 
$A_\pi x <> \mathbf{0}$. 

Using Lemma \ref{L:4} $(i)$ $A \pi^{-1}(x) = A_\pi x <> \mathbf{0}$ 
and we apply the result of the first part to $\pi^{-1} (x)$. Now using 
Lemma \ref{L:4} $(ii)$ we obtain
\begin{align*}    
A^T_\pi sgn ( A_\pi x ) & = \pi ( A^T sgn A \pi^{-1} ( x ) ) \\
                                   & = \pi ( 2 \pi^{-1}(x) - (n+1) \mathbf{1}_{n^2} ) \\
                                   & = 2 x - (n+1) \mathbf{1}_{n^2}
\end{align*}
and this completes the proof.
\end{proof}

This result can be applied to the generalized Sudoku problem.

\begin{theorem}      \label{T:2}
Let $n \ge 2$ and let $x \in \mathds{Z}^{n^2}$ be a solution of the generalized 
Sudoku problem. Then
\[
x = \frac{1}{2} ( A_{\pi_r}^T sgn (A_{\pi_r} x) + (n+1) \mathbf{1}_{n^2} ).
\]
for each $r=1, 2, 3$.
\end{theorem}
\begin{proof}
This follows immediately from Theorem \ref{T:1} and the definition 
of the generalized Sudoku problem.
\end{proof}

\begin{example}
We consider the case $n=3$, which implies $s(n) = 3$ and consider a solution 
$x \in \mathds{Z}^{3^2}$ which is depicted in this square:
\begin{align*}
\includegraphics{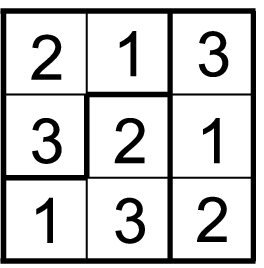}
\end{align*}
The matrix $A_{\pi_3}$ in this case looks like
\[
{
\renewcommand{\arraystretch}{0.5}     
\small                                                 
\begin{bmatrix}    
1 & -1 & 0 & 0 & 0 & 0 & 0 & 0 & 0\\
1 & 0 & 0 & -1 & 0 & 0 & 0 & 0 & 0\\
0 & 1 & 0 & -1 & 0 & 0 & 0 & 0 & 0\\

0 & 0 & 0 & 0 & 1 & 0 & -1 & 0 & 0\\
0 & 0 & 0 & 0 & 1 & 0 & 0 & -1 & 0\\
0 & 0 & 0 & 0 & 0 & 0 & 1 & -1 & 0\\

0 & 0 & 1 & 0 & 0 & -1 & 0 & 0 & 0\\
0 & 0 & 1 & 0 & 0 & 0 & 0 & 0 & -1\\
0 & 0 & 0 & 0 & 0 & 1 & 0 & 0 & -1
\end{bmatrix},
}
\]
$sgn( A_{\pi_3} x ) = ( 1, -1, -1, 1, -1 -1, 1, 1, -1)^T$ and finally
\[
\frac{1}{2} ( A_{\pi_3}^T sgn( A_{\pi_3} x ) + 4 \cdot \mathbf{1}_{3^2} )
\]
\[
= \frac{1}{2} (0, -2, 2, 2, 0, -2, -2, 2, 0)^T + 2 \cdot \mathbf{1}_{9} = (2, 1, 3, 3, 2, 1, 1, 3, 2)^T = x .
\]
\end{example}

Each solution of the generalized Sudoku problem coincides with the givens. This can be used in a
reformulation of Theorem \ref{T:2}. The brackets with an index in the subsequent corollary denote 
the corresponding component of a vector.

\begin{corollary}
Let $n \ge 2$ and let $x \in \mathds{Z}^{n^2}$ be a solution of the generalized 
Sudoku problem. Then
\[
g_{i_l} = \frac{1}{2} [ A_{\pi_r}^T sgn (A_{\pi_r} x ) + (n+1) \mathbf{1}_{n^2} ]_{i_l}
\]
for each $l = 1, \ldots, k$ and each $r = 1, 2, 3$.
\end{corollary}

                        %
                        %

\end{document}